\documentclass[a4paper,english]{amsart}
\usepackage{amssymb}
\usepackage{mathrsfs}
\usepackage{amsmath, amsfonts, vmargin, enumerate}
\usepackage{graphicx}
\usepackage{color}
\usepackage{verbatim}
\usepackage{amsthm}
\usepackage{latexsym, bm}
\usepackage{mathabx}
\usepackage[unicode=true,
bookmarks=false,
breaklinks=false,pdfborder={0 0 1},backref=false,colorlinks=false]
{hyperref}
\hypersetup{
	colorlinks,linkcolor=blue,anchorcolor=blue,citecolor=blue}

\usepackage{euscript}
\usepackage{dsfont}

\input xypic

\xyoption {all}

\makeatletter

\numberwithin{equation}{section}
\numberwithin{figure}{section}
\usepackage{enumitem}		% customizable list environments

\theoremstyle{plain}
\newtheorem{thm}{Theorem}[section]
\theoremstyle{plain}
\newtheorem{prop}[thm]{Proposition}
\theoremstyle{remark}
\newtheorem{rem}[thm]{Remark}

\theoremstyle{plain}
\newtheorem{cor}[thm]{Corollary}
\theoremstyle{plain}

\theoremstyle{plain}
\newtheorem{lem}[thm]{Lemma}
\theoremstyle{definition}
\newtheorem{example}[thm]{Example}
\theoremstyle{definition}
\newtheorem{defn}[thm]{Definition}

\renewcommand{\S}{\mathcal{S}}
\newcommand{\M}{\mathcal{M}}
\newcommand{\id}{\textnormal{id}}
\renewcommand{\G}{\mathbb{G}}
\newcommand{\F}{\mathcal{F}}
\newcommand{\m}{\mathfrak{m}}
\newcommand{\n}{\mathfrak{n}}
\newcommand{\I}{\mathcal{I}}
\newcommand{\T}{\mathbb{T}}
\newcommand{\Z}{\mathbb{Z}}
\newcommand{\Irr}{\mathrm{Irr}}
\newcommand{\Pol}{\mathrm{Pol}}
\newcommand{\tr}{\text{Tr}}
\usepackage{accents}
\newlength{\dhatheight}

\begin{document}
	
	\title{$L_p$-$L_q$ Fourier multipliers on locally compact quantum groups}
	\author{Haonan Zhang}

	\address{Institute of Science and Technology Austria (IST Austria),
		Am Campus 1, 3400 Klosterneuburg, Austria}
	
\email{haonan.zhang@ist.ac.at}
	
	%\subjclass[2010]{Primary: 43A99 46L51.  Secondary: 20G42 43A32.}
	
	
	\maketitle
	
	\begin{abstract}
		Let $\mathbb{G}$ be a locally compact quantum group with dual $\widehat{\mathbb{G}}$. Suppose that the left Haar weight $\varphi$ and the dual left Haar weight $\widehat{\varphi}$ are tracial, e.g. $\mathbb{G}$ is a unimodular Kac algebra. We prove that for $1<p\le 2 \le q<\infty$, the Fourier multiplier $m_{x}$	is bounded from $L_p(\widehat{\mathbb{G}},\widehat{\varphi})$ to $L_q(\widehat{\mathbb{G}},\widehat{\varphi})$ whenever the symbol $x$ lies in $L_{r,\infty}(\mathbb{G},\varphi)$, where $1/r=1/p-1/q$. Moreover, we have 
		\begin{equation*}
		\|m_{x}:L_p(\widehat{\mathbb{G}},\widehat{\varphi})\to L_q(\widehat{\mathbb{G}},\widehat{\varphi})\|\le c_{p,q} \|x\|_{L_{r,\infty}(\mathbb{G},\varphi)},
		\end{equation*}
		where $c_{p,q}$ is a constant depending only on $p$ and $q$. This was first proved by H\"ormander \cite{Hormander1960} for $\mathbb{R}^n$, and was recently extended to more general groups and quantum groups. Our work covers all these results and the proof is simpler. In particular, this also yields a family of $L_p$-Fourier multipliers over discrete group von Neumann algebras. A similar result for $\mathcal{S}_p$-$\mathcal{S}_q$ Schur multipliers is also proved.
		
	\end{abstract}
	
	\section{Introduction}
	
	For a nice function $f:\mathbb{R}^n\to \mathbb{C}$, let $\widehat{f}$ denote its Fourier transform, i.e.
	$$\widehat{f}(\xi)=\int_{\mathbb{R}^n}f(x)e^{-2\pi i\langle x,\xi\rangle}dx,$$ 
	where $\langle\cdot,\cdot\rangle$ denotes the Euclidean inner product on $\mathbb{R}^n$.
	Given $p,q>0$ and $\phi:\mathbb{R}^n\to\mathbb{C}$, the operator $m_{\phi}$ defined via 
	\begin{equation*}
	\widehat{m_{\phi} f}(\xi)=\phi(\xi)\widehat{f}(\xi),~~\xi\in\mathbb{R}^n,
	\end{equation*}
	is called an $L_p$-$L_q$ \emph{Fourier multiplier} if it is bounded from $L_p(\mathbb{R}^n)$ to $L_q(\mathbb{R}^n)$. When $p=q$, it is called an $L_p$-\emph{Fourier multiplier}  for short. The function $\phi$ is called the \emph{symbol} of the Fourier multiplier $m_\phi$.
	
	 H\"ormander proved the following $L_p$-$L_q$ Fourier multipliers theorem:
	
	\begin{thm}\cite[Theorem 1.11]{Hormander1960}\label{thm:Hormander's multiplier}
		Let $1<p\leq 2\leq q<\infty$ and $1/r=1/p-1/q$. Then we have
		\begin{equation*}
		\Vert m_{\phi}:L_p(\mathbb{R}^n)\to L_q(\mathbb{R}^n)\Vert\precsim_{p,q}\Vert \phi\Vert_{L_{r,\infty}(\mathbb{R}^n)}.
		\end{equation*}
	\end{thm}

	Here $L_{p,\infty}(\mathbb{R}^n)$ denotes the usual weak $L_p$-space. Throughout this paper, $C_1\precsim C_2$ always means $C_1\leq c C_2$ for some positive constant $c<\infty$. We write $C_1\precsim_p C_2$ if the constant $c=c_p$ is dependent of $p$. To prove Theorem \ref{thm:Hormander's multiplier}, H\"ormander used the following Paley-type inequalities.
	
	\begin{thm}\cite[Theorem 1.10]{Hormander1960}\label{thm:Hormander's Paley}
	For $1<p\leq2$ and  $1/s=2/p-1$, we have
	\begin{equation*}
	\Vert f\widehat{g}\Vert_{L_{p}(\mathbb{R}^n)}\precsim_{p} \Vert f\Vert_{L_{s,\infty}(\mathbb{R}^n)}\Vert g\Vert_{L_{p}(\mathbb{R}^n)}.
	\end{equation*}
	\end{thm}
	
	Both theorems have been generalized to compact Lie groups by Akylzhanov, Nursultanov and Ruzhansky \cite{Ruzhansky2016liegroup}, to locally compact separable unimodular groups by Akylzhanov and Ruzhansky \cite{Ruzhansky2016locallycompact} and to compact quantum groups of Kac type by Akylzhanov, Majid and Ruzhansky \cite{Ruzhansky2018compactquantumgroup}. Theorem \ref{thm:Hormander's Paley} for compact quantum groups of Kac type was also shown by Youn \cite{Sang-Gyun2018hardylittlewood}. All their proofs go back to H\"ormander \cite{Hormander1960}. 

	 Our first  result is a generalization of Theorem \ref{thm:Hormander's multiplier} to locally compact quantum groups $\G$ whose left Haar weight $\varphi$ and the dual left Haar weight $\widehat{\varphi}$ are both tracial. Our proof is slightly simpler and does not require Paley-type inequalities.  
	
	\begin{thm}\label{thm:main reulst Lp-Lq multiplier}
		Let $1<p\leq2\leq q<\infty$. Let $\G=(\M,\Delta,\varphi,\psi)$ be a locally compact quantum group with its dual $\widehat{\G}=(\widehat{\M},\widehat{\Delta},\widehat{\varphi},\widehat{\psi})$. Suppose that $\varphi$ and $\widehat{\varphi}$ are both tracial. Then for each $x\in L_{r,\infty}(\G,\varphi)$ with $1/r=1/p-1/q$, $m_x$ is an $L_p$-$L_q$ Fourier multiplier such that
		$$\| m_x:L_p(\widehat{\G},\widehat{\varphi})\to L_q(\widehat{\G},\widehat{\varphi})\|
		\precsim_{p,q}\| x\|_{L_{r,\infty}(\G,\varphi)}.$$
	\end{thm}
	See Sections \ref{sec:preliminaries} and \ref{sec:fourier_transform} for the corresponding definitions. We will not deduce Theorem \ref{thm:main reulst Lp-Lq multiplier} from Paley-type inequalities, but we may still extend Theorem \ref{thm:Hormander's Paley} to locally compact quantum groups with a slightly simpler proof. This is our second result.
	
	\begin{thm}\label{thm:main result-Paley type}
		Let $1<p\leq2$. Let $\G=(\M,\Delta,\varphi,\psi)$ be a locally compact quantum group with its dual $\widehat{\G}=(\widehat{\M},\widehat{\Delta},\widehat{\varphi},\widehat{\psi})$. Suppose that $\varphi$ and $\widehat{\varphi}$ are both tracial. Then we have
		\begin{equation*}
		\| a\F(x)\|_{L_p(\widehat{\G},\widehat{\varphi})}
		\precsim\| a\|_{L_{s,\infty}(\widehat{\G},\widehat{\varphi})}\| x\|_{L_p(\G,\varphi)},
		\end{equation*}
		for all $a\in L_{s,\infty}(\widehat{\G},\widehat{\varphi})$ and $x\in L_p(\G,\varphi)$, where $1/s=2/p-1$.
	\end{thm}
	
	Here $\F$ denotes the Fourier transform; see Section \ref{sec:fourier_transform} for the definition. If furthermore, the dual quantum group $\widehat{\G}$ is compact, then Theorem \ref{thm:main reulst Lp-Lq multiplier} gives a sufficient conditions for $L_p$-Fourier multipliers on $\G$.  This is our third result.
	
	\begin{thm}\label{thm:main result-Lp multiplier}
		Fix $1<p<\infty$ and $1/p^*=|1/2-1/p|$. Let $\G$ be a compact quantum group of Kac type with Haar state $h$. Let $\widehat{\G}$ be its dual with dual Haar weight $\widehat{h}$. Let $\widehat{\F}$ be the Fourier transform over $\widehat{\G}$. Then for any $a=(a_{\pi})_{\pi\in\Irr(\G)}\in\widehat{\G}$, the Fourier multiplier $m_a:\widehat{\F}(b)\mapsto \widehat{\F}(ab)$ satisfies
		\begin{equation*}
		\|m_a:L_p(\G,h)\to L_p(\G,h)\|\precsim_{p}\|a\|_{\ell_{p^\ast,\infty}(\widehat{\G},\widehat{h})}.
		\end{equation*}
	\end{thm}

An interesting family of such examples is obtained by choosing $\G=\widehat{G}$ as the group von Neumann algebra of a discrete group $G$. 

\begin{cor}\label{cor:group vna}
	 For any discrete group $G$ let $\widehat{G}$ be the group von Neumann algebra equipped with the canonical tracial state $\tau$. Then for any $\phi\in\ell_{p^*,\infty}(G)$ with $1/p^*=|1/2-1/p|$, the Fourier multiplier $m_{\phi}:L_p(\widehat{G},\tau)\to L_{p}(\widehat{G},\tau),\lambda_g\mapsto \phi(g)\lambda_g$ extends to a bounded map such that 
	\begin{equation*}
	\|m_{\phi}:L_p(\widehat{G},\tau)\to L_{p}(\widehat{G},\tau)\|\precsim_{p} \|\phi\|_{\ell_{p^*,\infty}(G)},
	\end{equation*}
	where $\lambda$ is the left regular representation of $G$.
\end{cor}

An analogue of $L_p$-$L_q$ Fourier multipliers theorem is also valid for Schur multipliers. We use $\S_p(H)$ to denote the Schatten $p$-classes $L_p(B(H))$.

\begin{thm}\label{thm:schur}
	Let $1< p\le 2\le q<\infty$ and $1/r=1/p-1/q$. Let $X$ be a set. The Schur multiplier $A:(x_{ij})_{i,j\in X}\mapsto (a_{ij}x_{ij})_{i,j\in X}$ satisfies
	\begin{equation*}
	\|A:\S_p(\ell_{2}(X))\to \S_q(\ell_{2}(X))\|\precsim_{p,q} \|a\|_{\ell_{r,\infty}(X\times X)},
	\end{equation*}
	where on the right hand side $a=(a_{ij})_{i,j\in X}$ is identified as an element in $\mathbb{C}^{X\times X}$.
\end{thm}

The paper is organized as follows. In Section \ref{sec:preliminaries} we recall basic knowledge of locally compact quantum groups and noncommutative ($L_p$- and) Lorentz spaces. Section \ref{sec:fourier_transform} presents the Fourier transforms on locally compact quantum groups and (complex, real) Hausdorff--Young inequalities. In Section \ref{sec:proof}, we prove the main results and give some examples.
	
	\section{Preliminaries}\label{sec:preliminaries}
	In this section we collect some necessary preliminaries of locally compact quantum groups, noncommutative $L_p$-spaces and noncommutative Lorentz spaces.
	
	\subsection{Noncommutative $L_p$-spaces and Lorentz spaces associated with a semifinite von Neumann algebra}
	We concentrate ourselves on noncommutative $L_p$-spaces associated with semifinite von Neumann algebras, which were first laid out in the early 50's by Segal \cite{Segal1953abstractintergration} and Dixmier\cite{Dixmier1953}. The noncommutative Lorentz spaces will be treated at the same time. We refer to \cite{PisierXu2003LP} for more discussions.
	
	Let $\M$ be a semifinite von Neumann algebra equipped with a normal semifinite faithful ($n.s.f.$) trace $\tau$. Denote by $\M^+$ the positive cone of $\M$. Let $\mathcal{S}^+$ denote the set of all $x\in\M^+$ such that $\tau(\text{supp}(x))<\infty$, where $\text{supp}(x)$ denotes the support of $x$. Let $\mathcal{S}$ be the linear span of $\mathcal{S}^+$. Then $\mathcal{S}$ is a weak*-dense *-subalgebra of $\M$. Given $0<p<\infty$, we define
	\begin{equation*}
	\Vert x\Vert_p:=[\tau(|x|^p)]^{\frac{1}{p}},~~x\in\mathcal{S},
	\end{equation*}
	where $|x|=(x^*x)^{\frac{1}{2}}$ is the modulus of $x$. Then $(\mathcal{S},\Vert\cdot\Vert_p)$ is a normed (or quasi-normed for $p<1$) space. Its completion is called \emph{noncommutative $L_p$-space associated with $(\M,\tau)$}, denoted by $L_p(\M,\tau)$ or simply by $L_p(\M)$. As usual, we set $L_\infty(\M,\tau)=\M$ equipped with the operator norm.
	
	For $1\le p<\infty$, the dual space of $L_p(\M)$ is $L_{p'}(\M)$ with respect to the duality
	$$\langle x,y\rangle :=\tau(xy),~~,x\in L_p(\M),y\in L_{p'}(\M).$$
	In particular, $L_1(\M)$ is identified with $\M_*$ via the map $j(x):=\tau(x\cdot),x\in L_1(\M)$. 
	
	The elements in $L_p(\M)$ can be viewed as closed densely defined operators on $H$ ($H$ being the Hilbert space on which $\M$ acts). A linear closed operator $x$ is said to be \textit{affiliated with} $\M$ if it commutes with all unitary elements in $\M'$, i.e. $xu=ux$ for any unitary $u\in\M'$. Note that $x$ can be unbounded on $H$. An operator $x$ affiliated with $\M$ is said to be \textit{measurable with respect to $(\M,\tau)$}, or simply \textit{measurable} if for any $\varepsilon>0$, there exists a projection $e\in\M$ such that
	$$e(H)\subset\mathcal{D}(x) \text{ and }  \tau(e^\perp)\leq\varepsilon,$$
	where $e^\perp=1-e$. We denote by $L_0(\M,\tau)$, or simply $L_0(\M)$ the family of measurable operators. For $x\in L_0(\M,\tau)$, we define the \textit{distribution function} of $x$
	\begin{equation*}
	\lambda_s(x):=\tau(\chi_{(s,\infty)}(|x|)),~~s\geq0,
	\end{equation*}
	where $\chi_{(s,\infty)}(|x|)$ is the spectral projection of $|x|$ corresponding to the interval $(s,\infty)$, and define the \textit{generalized singular numbers} of $x$
	\begin{equation*}
	\mu_t(x):=\inf\{s>0:\lambda_s(x)< t\},~~t\geq 0.
	\end{equation*}
	Similar to the classical case, for $0<p<\infty, 0<q\leq\infty$, the \emph{noncommutative Lorentz space} $L_{p,q}(\M)$ is defined as the collection of all measurable operators $x$ such that
	\begin{equation*}
	\Vert x\Vert_{p,q}:=\left(\int_{0}^{\infty}(t^{\frac{1}{p}}\mu_t(x))^q\frac{dt}{t}\right)^{\frac{1}{q}}<\infty.
	\end{equation*}
	Clearly, $L_{p,p}(\M)=L_p(\M)$ with $\|\cdot\|_{p,p}=\|\cdot\|_{p}$. The space $L_{p,\infty}(\M)$ is usually called the \emph{weak $L_p$-space}, $0<p<\infty$, and one defines
	\begin{equation*}
	\Vert x\Vert_{p,\infty}:=\sup_{t>0}t^{\frac{1}{p}}\mu_t(x).
	\end{equation*}
	Like the classical $L_p$-spaces, noncommutative $L_p$-spaces behave well with respect to interpolation. Our reference for interpolation theory is \cite{BerghLofstrom1976interpolationspaces}. Let $1\leq p_0\leq p_1\leq\infty$, $1\leq q\leq\infty$ and $0<\theta<1$. Suppose 
	$$\frac{1}{p}=\frac{1-\theta}{p_0}+\frac{\theta}{p_1}.$$
	Then it is well-known that \cite[Section 2]{PisierXu2003LP}
	\begin{equation*}
	(L_{p_0}(\M),L_{p_1}(\M))_\theta=L_p(\M) \text{  (with equal norms)  }
	\end{equation*}
	and 
	\begin{equation}\label{eq:real interpolation}
	(L_{p_0}(\M),L_{p_1}(\M))_{\theta,q}=L_{p,q}(\M)\text{  (with equivalent quasi-norms)},
	\end{equation}
	 where $(\cdot,\cdot)_\theta$ and $(\cdot,\cdot)_{\theta,q}$ denote respectively the complex and real interpolation methods.
	
	We formulate here some properties that we will use in this paper. For the proofs we refer to \cite{FackKosaki1986sigularnumbers} and \cite{Grafakos2014classical}.
		
	\begin{lem}\label{lem:properties of Lorentz spaces}
		Let $\M$ be a von Neumann algebra equipped with a $n.s.f.$ trace $\tau$. We have
		\begin{enumerate}
			\item $\mu_{s+t}(xy)\leq\mu_s(x)\mu_t(y)$ for all $s,t\geq 0$ and $x,y\in L_0(\M)$;
			\item for any $1< p,q<\infty$ and $q<r\leq\infty$, 
			\begin{equation}\label{ineq:lorentz space}
			\Vert x\Vert_{p,r}\precsim_{p,q,r}\Vert x\Vert_{p,q},~~x\in L_{p,q}(\M),
			\end{equation}
			where the constant is $c_{p,q,r}=(q/p)^{\frac{1}{q}-\frac{1}{r}}$.
		\end{enumerate}	
	\end{lem}
	
	H\"older type inequalities hold on noncommutative Lorentz space. We only present here a special case that is enough for our use. We give a proof here for reader's convenience.
	
	\begin{lem}
		Let $0<p_0,p_1<\infty$, $0<q<\infty$, and $1/p=1/p_0+1/p_1$.
		Let $\M$ be a von Neumann algebra equipped with a $n.s.f.$ trace $\tau$. Then we have
		\begin{equation}\label{ineq:Holder}
		\Vert xy\Vert_{p,q}\precsim_{p}\Vert x\Vert_{p_0,\infty}\Vert y\Vert_{p_1,q},~~x\in L_{p_0,\infty}(\M),y\in L_{p_1,q}(\M),
		\end{equation}
		where the constant is $c_p=2^{\frac{1}{p}}$.
	\end{lem} 
	
	\begin{proof}
		From Lemma \ref{lem:properties of Lorentz spaces}(1) and the definition of $\|\cdot\|_{p_0,\infty}$, it follows that
		\begin{align*}
		\Vert xy\Vert_{p,q}
		&=\left(\int_{0}^{\infty}\left(t^{\frac{1}{p}}\mu_t(xy)\right)^q\frac{dt}{t}\right)^{\frac{1}{q}}\\
		&=2^{\frac{1}{p}}\left(\int_{0}^{\infty}\left(t^{\frac{1}{p}}\mu_{2t}(xy)\right)^q\frac{dt}{t}\right)^{\frac{1}{q}}\\
		&\le 2^{\frac{1}{p}}\left(\int_{0}^{\infty}\left(t^{\frac{1}{p_0}}\mu_t(x)\cdot t^{\frac{1}{p_1}}\mu_t(y)\right)^q\frac{dt}{t}\right)^{\frac{1}{q}}\\
		&=2^{\frac{1}{p}}\Vert x\Vert_{p_0,\infty}\Vert y\Vert_{p_1,q}.\qedhere
		\end{align*}
	\end{proof}

	\subsection{Locally compact quantum groups}
	In this subsection we recall the definition of locally compact quantum groups in the sense of Kustermans and Vaes \cite{KustermansVaes2000lcqg,KV03}. See also the notes \cite{Caspers2017notes}. We shall mainly work with the von Neumann algebraic version. For any $n.s.f.$ weight $\varphi$ on a von Neumann algebra $\M$, we set
	\begin{equation*}
	\n_{\varphi}:=\{x\in\M:\varphi(x^*x)<\infty\},~~\m_{\varphi}:=\n_{\varphi}^{*}\n_{\varphi}.
	\end{equation*}
	A \emph{locally compact quantum group} $\G=(\M,\Delta,\varphi,\psi)$ consists of
	\begin{enumerate}
		\item a von Neumann algebra $\M$;
		\item a normal, unital, *-homomorphism $\Delta:\M\to\M\overline{\otimes}\M$ such that $$(\Delta\otimes \id)\Delta=(\id\otimes\Delta)\Delta;$$
		\item a $n.s.f.$ weight $\varphi$ which is left invariant $$\varphi[(\omega\otimes\id)\Delta(x)]=\varphi(x)\omega(1),~~\omega\in \M_*^+, x\in\m^+_\varphi;$$
		\item a $n.s.f.$ weight $\psi$ which is right invariant $$\psi[(\id\otimes\omega)\Delta(x)]=\psi(x)\omega(1),~~\omega\in \M_*^+,x\in\m^+_{\psi};$$
	\end{enumerate} 
	where $\overline{\otimes}$ denotes the von Neumann algebra tensor product and $\id$ denotes the identity map. The normal, unital, *-homomorphism $\Delta$ is called \textit{comultiplication} on $\M$, $\varphi$ is called \textit{left Haar weight} and $\psi$ is called \textit{right Haar weight}.
	
	\begin{example}\label{exp:locally compact group}
		Let $G$ be a locally compact group. Then $(L_\infty(G,\mu),\Delta,\mu,\nu)$ is a locally compact quantum group, where $\Delta:L_\infty(G,\mu)\to L_\infty(G,\mu)\overline{\otimes}L_\infty(G,\mu)\simeq L_\infty(G\times G,\mu\times \mu)$ is given by $\Delta(f)(s,t)=f(st),s,t\in G$, and $\mu,\nu$ are the left and right Haar measures on $G$, respectively.
	\end{example}
	
	Given a locally compact quantum group $\G=(\M,\Delta,\varphi,\psi)$, we now define its dual  $\widehat{\G}=(\widehat{\M},\widehat{\Delta},\widehat{\varphi},\widehat{\psi})$ that is also a locally compact quantum group. For this, we equip $\n_{\varphi}$ with the inner product 
	$$\langle x,y\rangle=\varphi(y^\ast x),$$
	and denote by $H_{\varphi}$ the induced Hilbert space after completion. For any $x\in\n_{\varphi}\subset\M$ we write $\Lambda_{\varphi}(x)$ for the corresponding element in $H_{\varphi}.$ For any $x\in \M$, $\pi_{\varphi}(x)$ denotes the bounded operator over $H_{\varphi}$ such that $\pi_{\varphi}(x)\Lambda_{\varphi}(y)=\Lambda_{\varphi}(xy)$.
	So $(H_{\varphi},\pi_{\varphi},\Lambda_{\varphi})$ is the GNS representation of $\varphi$. We omit the subscript $\varphi$ in the sequel whenever there is no ambiguity. Assume that $\M$ acts on $H_\varphi$ with its predual $\M_*$. The \textit{multiplicative unitary} of $\G$ is the unitary operator $W$ on $H_\varphi\otimes H_\varphi$ such that
	\[
	W^*(\Lambda(x)\otimes\Lambda(y))=(\Lambda\otimes\Lambda)(\Delta(y)(x\otimes 1)),~~x,y\in \n_\varphi.
	\]
	 It implements the comultiplication:
	\[
	\Delta(x)=W^\ast (1\otimes x)W,~~x\in\M.
	\]
	For any $\omega\in \M_*$, define
	\begin{equation}\label{eq:fourier representation}
	\lambda(\omega):=(\omega\otimes\id)W.
	\end{equation}
	Then the underlying von Neumann algebra of $\widehat{\G}$ is defined as $\widehat{\M}:=\lambda(\M_*)''\subset B(H_\varphi)$. The comultiplication of $\widehat{\G}$ is given by
	$$\widehat{\Delta}(x)=\widehat{W}(1\otimes x)\widehat{W}^*,~~x\in\widehat{\M},$$
	where $\widehat{W}=\Sigma W^*\Sigma$ is the multiplicative unitary on $\widehat{\G}$ with $\Sigma$ being the flip on $H_\varphi\otimes H_\varphi$, i.e. $\Sigma(\xi\otimes\eta)=\eta\otimes\xi$.
	
	To define the dual left Haar weights, set
	\begin{equation*}
	\I:=\{\omega\in\M_*:\exists C>0 \text{ such that } |\omega(x^*)|\leq C\Vert\Lambda(x)\Vert,~x\in\n_\varphi\}.
	\end{equation*}
	By the Riesz representation theorem, there exists unique $\xi(\omega)\in H_\varphi$ such that $$\omega(x^*)=\langle\xi(\omega),\Lambda(x)\rangle,~~x\in\n_\varphi.$$ 
	Then the \emph{dual left Haar weight} $\widehat{\varphi}$ is defined to be the unique $n.s.f.$ weight on $\widehat{\M}$ with the GNS representation $(H,\iota,\widehat{\Lambda})$ such that $\lambda(\I)$ is a $\sigma$-strong*-norm core for $\widehat{\Lambda}$ and $\widehat{\Lambda}(\lambda(\omega))=\xi(\omega)$ for all $\omega\in\I$.
	Thus we have 
	\begin{equation}\label{eq:Parseval's relation}
	\widehat{\varphi}(\lambda(\omega)^*\lambda(\omega))
	=\langle\widehat{\Lambda}(\lambda(\omega)),\widehat{\Lambda}(\lambda(\omega))\rangle,~~\omega\in\I.
	\end{equation}
	The \emph{dual right Haar weight} $\widehat{\psi}$ can be defined in a similar way, which we will not do here. Then $\widehat{\G}=(\widehat{\M},\widehat{\Delta},\widehat{\varphi},\widehat{\psi})$ forms a locally compact quantum group. Constructing the dual $\widehat{\widehat{\G}}$ of $\widehat{\G}$, the Pontryagin duality says $\widehat{\widehat{\G}}=\G$. Furthermore, we have $\widehat{\widehat{\Lambda}}=\Lambda$.
	
	In this paper we are interested in locally compact quantum groups $\G$ on which both left Haar weight $\varphi$ and dual left Haar weight $\widehat{\varphi}$ are tracial. We close this subsection with some examples of locally compact quantum groups of this type.
	
	\begin{example}[Unimodular Kac algebras]
		We refer to \cite{ES92kac} for more about Kac algebras. Here we only remark that \emph{unimodular Kac algebras} are locally compact quantum groups $\G=(\M,\Delta,\varphi,\psi)$  for which $\varphi=\psi$ is tracial. If a Kac algebra $\G$ is unimodular, then so is its dual $\widehat{\G}$ \cite[Proposition 6.1.4]{ES92kac}. We give more concrete examples in the following. 
	\end{example}
	
	\begin{example}[Locally compact unimodular groups]
		Let $G$ be a locally compact unimodular group with $\mu$ being left (also right) Haar measure. Then $\G=(L_{\infty}(G,\mu),\Delta,\mu,\mu)$ is a locally compact quantum group, as we have seen in Example \ref{exp:locally compact group}. Clearly, its left Haar weight is tracial. From \cite{Kunze1958fouriertransform}, on its dual quantum group $\widehat{\G}=(\widehat{G},\widehat{\Delta},\widehat{\mu},\widehat{\mu})$, the left (right) dual Haar weight is tracial. Here $\widehat{G}$ is the von Neumann algebra acting on $L_2(G,\mu)$ generated by all $\lambda(f),f\in L_1(G,\mu)$, where $\lambda(f)$ is the convolution operator: $\lambda(f)g=f*g,g\in L_2(G,\mu)$. The multiplicative unitary $W$ acts on $L_2(G,\mu)\otimes L_2(G,\mu)\simeq L_2(G\times G,\mu\times \mu)$ as
		$$WF(s,t)=F(s,s^{-1}t).$$
	\end{example}

	\begin{example}[Compact quantum groups of Kac type]
		A \emph{compact quantum group} is a locally compact quantum group $\G$ such that the left Haar weight is finite, i.e. $\varphi(1)<\infty$. This agrees with Woronowicz's definition of compact quantum groups \cite{Woronowicz1998notes}, which we shall now recall. A \emph{compact quantum group} consists of a pair $\mathbb{G}=(A,\Delta)$, where $A$ is a unital C*-algebra and $\Delta$ is a unital $^\ast$-homomorphism from $A$ to $A\otimes A$ such that 
		\begin{enumerate}
			\item $(\Delta\otimes\id)\Delta=(\id\otimes\Delta)\Delta$;
			\item $\{\Delta(a)(1\otimes b):a,b\in A\}$ and $\{\Delta(a)(b\otimes1):a,b\in A\}$ are linearly dense in $A\otimes A$.
		\end{enumerate}
		Here $A\otimes A$ is the minimal C*-algebra tensor product. Any compact quantum group admits a unique \emph{Haar state}, i.e. a state $h$ on $A$ that is both left and right invariant:
		\begin{equation*}
		(h\otimes\id)\Delta(a)=h(a)1=(\id\otimes h)\Delta(a),~~a\in A.
		\end{equation*}
		Consider an element $u\in A\otimes B(H)$ with $\dim H=n$. By identifying $A\otimes B(H)$ with $M_n(A)$ we can write $u=[u_{ij}]_{i,j=1}^{n}$, where $u_{ij}\in A$. The matrix $u$ is called an \textit{n-dimensional representation} of $\mathbb{G}$ if we have
		\[
		\Delta(u_{ij})=\sum_{k=1}^{n}u_{ik}\otimes u_{kj},~~i,j=1,\dots,n.
		\]
		A representation $u$ is called \textit{unitary} if $u$ is unitary as an element in $M_n(A)$, and \textit{irreducible} if the only matrices $T\in M_n(\mathbb{C})$ such that $uT=Tu$ are multiples of identity matrix. Two representations $u,v\in M_n(A)$ are said to be \textit{equivalent} if there exists an invertible matrix $T\in M_n(\mathbb{C})$ such that $Tu=vT$. Denote by $\Irr(\mathbb{G})$ the set of equivalence classes of irreducible unitary representations of $\mathbb{G}$. For each $\pi\in\Irr(\mathbb{G})$, denote by $u^\pi\in A\otimes B(H_\pi)$ a representative of the class $\pi$, where $H_\pi$ is the finite-dimensional Hilbert space on which $u^\pi$ acts. In the sequel we write $n_\pi=\dim H_\pi$. Denote $\Pol(\mathbb{G})=\text{span} \left\{u^\pi_{ij}:1\leq i,j\leq n_\pi,\pi\in\Irr(\mathbb{G})\right\}$. This is a dense subalgebra of $A$. 
		
		The dual of a compact quantum group $\G$ is a discrete quantum group $\widehat{\G}=(\widehat{A},\widehat{\Delta},\widehat{h}_{\text{L}},\widehat{h}_{\text{R}})$, where $\widehat{A}$ is the $c_0$-direct sum of matrix algebras
		\begin{equation*}
		\widehat{A}=\bigoplus_{\pi\in\Irr(\G)}B(H_{\pi}).
		\end{equation*}
		The dual left Haar weight $\widehat{h}_{\text{L}}$ and dual right Haar weight $\widehat{h}_{\text{R}}$ are not the same in general. A compact quantum group $\G$ is of \emph{Kac type} if the Haar state $h$ is tracial. In this case $\widehat{h}_{\text{L}}$ and $\widehat{h}_{\text{R}}$ coincide, which we denote by $\widehat{h}$ for short. It takes the following form
		\begin{equation*}
		\widehat{h}(a)=\sum_{\pi\in\Irr(\G)}d_{\pi}\tr (a_{\pi}).
		\end{equation*}
		
		The multiplicative unitary $W$ of $\G$ is
		$$W:=\bigoplus_{\pi\in\Irr(\G)} u^{\pi}.$$
		Then the Fourier transform $\F$ over $\Pol(\G)$ is given by 
		\begin{equation*}
		\F(x)=(h(\cdot x) \otimes \id)W=(\widehat{x}(\pi))_{\pi\in\Irr(\G)},
		\end{equation*}
		where $\widehat{x}(\pi)=(h(\cdot x)\otimes \id)(u^{\pi}) $.
		
		Classical compact groups are certainly compact quantum groups of Kac type (the commutative case). In the next we give another family of such quantum groups (the cocommutative case). There are also compact quantum groups of Kac type which are neither commutative nor cocommutative, e.g. \emph{free orthogonal quantum groups $O_N^+$} \cite{Wan95} and \emph{free permutation quantum groups} $S_N^+$ \cite{Wan98}. We will not explain here in detail.
		
	\end{example}
	
	\begin{example}[Discrete group von Neumann algebras]
		Let $G$ be a discrete group. Then $\G=(\ell_{\infty}(G),\Delta,\mu,\mu)$ is a locally compact quantum group with $\mu$ being the counting measure. Suppose that $\{\delta_g\}_{g\in G}$ is the canonical basis of $\ell_2(G)$. Then the left regular representation of $G$ is given through $\lambda:G\to B(\ell_2(G)),\lambda_g(\delta_h)=\delta_{gh}$. The group von Neumann algebra $\widehat{G}$ is the von Neumann algebra generated by $\lambda(g),g\in G$ in $B(\ell_2(G))$. Thus the dual quantum group of $G$ is $\widehat{\G}=(\widehat{G},\widehat{\Delta},\tau,\tau)$, where $\tau$ is a normal faithful tracial state defined by $\tau(x)=\langle\delta_e,x\delta_e\rangle$, where $e$ is the unit of $G$ and $\langle\cdot,\cdot\rangle$ is the inner product on $\ell_2(G)$.
	\end{example}

	\section{Fourier transform on locally compact quantum groups }\label{sec:fourier_transform}
	In the remaining part of the paper, unless otherwise stated, for any $1<p<\infty$, $p'$ always denotes the conjugate number of $p$, i.e. $1/p+1/p'=1$. $\G=(\M,\Delta,\varphi,\psi)$ always denotes a locally compact quantum group with dual $\widehat{\G}=(\widehat{\M},\widehat{\Delta},\widehat{\varphi},\widehat{\psi})$, where $\varphi$ and $\widehat{\varphi}$ are both tracial. We shall use $L_{p}(\G,\varphi)$ and $L_{p,q}(\G,\varphi)$ to denote $L_{p}(\M,\varphi)$ and $L_{p,q}(\M,\varphi)$, respectively. The same goes to $L_{p}(\widehat{\G},\widehat{\varphi})$ and $L_{p,q}(\widehat{\G},\widehat{\varphi})$.
	
	\subsection{A brief history}
	In this section we briefly recall the history of Fourier transform on locally compact quantum groups and its definition in our setting.
	
	Let $G$ be a locally compact abelian group with Haar measure $\mu$, then the Fourier transform of $f\in L_{1}(G,\mu)$ takes the form:
	\[
	\F(f)(\xi)=\widehat{f}(\xi)=\int_{G}f(s)\overline{\xi(s)}d\mu(s),~~\xi\in\widehat{G}.
	\]
	By choosing the dual Haar measure $\widehat{\mu}$ on $\widehat{G}$ suitably, the map $L_{1}(G,\mu)\cap L_{2}(G,\mu)\ni f\mapsto \widehat{f}\in L_{2}(\widehat{G},\widehat{\mu})$ is isometric and can be extended to an isometry between $L_{2}(G,\mu)$ and $L_{2}(\widehat{G},\widehat{\mu})$. This defines the Fourier transform of $f\in L_{2}(G,\mu)$. The definition of the Fourier transform of $f\in L_{p}(G,\mu)$ follows from the the famous Hausdorff--Young inequality, which states that for any $1\leq p\leq 2$ we have 
	\begin{equation}\label{ineq:H-Y abelian}
	\Vert\widehat{f}\Vert_{L_{p'}(\widehat{G},\widehat{\mu})}\leq \Vert f\Vert_{L_{p}(G,\mu)},~~f\in L_{p}(G,\mu).
	\end{equation} 
	
	It is natural to ask what the Fourier transform looks like for general locally compact groups and whether we still have \eqref{ineq:H-Y abelian} or not. The first breakthrough is due to Kunze \cite{Kunze1958fouriertransform}, who observed the following fact. Let $G$ be a locally compact abelian group as above. Let $\lambda(f)$ denote the left regular representation of $f\in L_{1}(G,\mu)$ on $L_{2}(G,\mu)$, which is an operator given by
	\[
	(\lambda(f)g)(s):=f*g(s)=\int_{G}f(t)g(t^{-1}s)d\mu(t),~~g\in L_{2}(G,\mu).
	\]
	Denote by $L_{f}$ the operator on $L_{2}(G,\mu)$ given by multiplying $f$. Since $\F$ turns convolution into multiplication, we have
	\[
	\F(\lambda(f)g)=\F(f)\F(g)=L_{\F(f)}\F(g),~~f\in L_{1}(G,\mu),~g\in L_{2}(G,\mu).
	\]
	Recall that $\F$ is unitary on $L_{2}(G,\mu)$, so $\lambda(f)$ is unitarily equivalent to the operator $L_{\F(f)}$. This suggests us to use $\lambda(f)$ as a substitute of $\F(f)$. From this Kunze defined the Fourier transform on locally compact unimodular groups $(G,\mu)$ and generalized Hausdorff--Young inequalities \eqref{ineq:H-Y abelian} to locally compact unimodular groups. The dual of $G$, still denoted by $\widehat{G}$, is no longer a group, but can be studied via the von Neumann algebra generated by $\lambda(L_{1}(G,\mu))$ in $B(L_{2}(G,\mu))$. It turns out that there is a canonical  trace $\widehat{\mu}$ on $\widehat{G}$, so $L_{p'}(\widehat{G},\widehat{\mu})$ is constructed in the sense of Diximier and Segal. Terp \cite{Terp2017fouriertransform} extended this approach to locally compact non-unimodular groups $G$. Her Fourier transform for $f\in L_{p}(G,\mu)$ is the operator on $L_{2}(G,\mu)$ given by $\lambda(f)\Delta^{\frac{1}{p'}}$, where $\mu$ is the left Haar measure and $\Delta$ is the modular function on $G$. Remark that $\Delta$ here is understood as a multiplication operator by $\Delta$. The dual $\widehat{G}$ of $G$ is not necessarily equipped with a trace. In this context we also have Hausdorff--Young inequalities, where $L_{p'}(\widehat{G})$ is the noncommutative $L_p$-space in the sense of Hilsum \cite{Hilsum1981} and Connes \cite{Connes1980spatial}. Finally the Hausdorff--Young inequalities were extended to locally compact quantum groups by Cooney \cite{Cooney2010hausdorffyoung} and Caspers \cite{Caspers2013fouriertransform}.
	
	In this paper we are concerned with the locally compact quantum group case, but the associated left Haar weight and dual left Haar weight are both tracial. This makes the definition of $L_p$-Fourier transform much simpler than those of Cooney and Caspers. Indeed, we can embed our noncommutative $L_p$-space $L_p(\G,\varphi)$ ($1<p<2$) into $L_1(\G,\varphi)+L_2(\G,\varphi)$ in a natural way. So we will not recall their approaches here. 
	
	\subsection{Fourier transform and Hausdorff--Young inequalities}
	This subsection does not contain any new results. See for example \cite{Caspers2013fouriertransform}. We collect the proofs here for reader's convenience. 
	
	\begin{prop}\label{prop:intersection of L1 and L2}
		We have $L_1(\G,\varphi)\cap L_2(\G,\varphi)=\I$.
	\end{prop}
	
	This holds for general locally compact quantum groups and should be understood under suitable embedding of $\I$, $L_1(\G,\varphi)$ and $L_2(\G,\varphi)$ into some Banach space \cite[Theorem 3.3]{Caspers2013fouriertransform}. We give a proof here when $\varphi$ is tracial, which is the case this paper concerns with. In such case, $\I$ should be understood as $j^{-1}(\I)$, where $j:L_1(\G,\varphi)\to\M_*, x\mapsto \varphi(x\cdot)$ is the isometry map.
	
	\begin{proof}[Proof of Proposition \ref{prop:intersection of L1 and L2} when $\varphi$ is tracial]
		By definition,
		$$j^{-1}(\I)=\{y\in L_1(\G,\varphi):\exists C<\infty \textnormal{ such that }|\varphi(x^*y)|\le C\|x\|_{L_2(\G,\varphi)}, x\in \n_{\varphi} \}$$
		Note that $\n_{\varphi}$ is dense in $L_2(\G,\varphi)$, by duality of $L_p$-spaces, we have
		\begin{align*}
		j^{-1}(\I)
		&=\{y\in L_1(\G,\varphi):\exists C<\infty \textnormal{ such that }|\varphi(x^*y)|\le C\|x\|_{L_2(\G,\varphi)}, x\in L_2(\G,\varphi) \}\\
		&=\{y\in L_1(\G,\varphi):\|y\|_{L_2(\G,\varphi)}<\infty\}\\
		&=L_1(\G,\varphi)\cap L_2(\G,\varphi).\qedhere
		\end{align*}
	\end{proof}

	Recall that $L_1(\G,\varphi)$ is identified with $\M_*$ via the map $j(x)= \varphi_x$, where $\varphi_x:=\varphi(x\cdot).$  Since $W$ is unitary, from \eqref{eq:fourier representation} we have 
	\begin{equation*}
	\|\lambda(\varphi_x)\|_{L_\infty(\widehat{\G},\widehat{\varphi})}
	\le \|\varphi_x\|_{\M_*}=\|x\|_{L_1(\G,\varphi)},~~x\in L_1(\G,\varphi).
	\end{equation*}
	We define the $L_1$-Fourier transform as $\F_1:=\lambda\circ j:L_1(\G,\varphi)\to L_\infty(\widehat{\G},\widehat{\varphi}),x\mapsto \lambda(\varphi_x)$, then it is a contraction:
	\begin{equation*}
	\|\F_1(x)\|_{L_\infty(\widehat{\G},\widehat{\varphi})}
	\le \|x\|_{L_1(\G,\varphi)},~~x\in L_1(\G,\varphi).
	\end{equation*}
	For the $L_2$-Fourier transform, we firstly define it as $\F_1$ on the intersection of $L_1(\G,\varphi)$ and $\n_{\varphi}$.
	By Proposition \ref{prop:intersection of L1 and L2}, for any $x\in L_1(\G,\varphi)\cap \n_{\varphi}$, $\varphi_x$ belongs to $\I$. Note that by definition of $\widehat{\Lambda}$, we have
	$$\langle \Lambda(x),\Lambda(y)\rangle
	=\varphi(y^*x)
	=\varphi_x(y^*)
	=\langle\widehat{\Lambda}(\lambda(\varphi_x)),\Lambda(y) \rangle,~~y\in \n_{\varphi}.$$
	Since $\n_{\varphi}$ is dense in  $L_2(\G,\varphi)$, we have 
	\begin{equation}\label{eq:lambda}
	\widehat{\Lambda}(\lambda(\varphi_x))=\Lambda(x),~~x\in L_1(\G,\varphi)\cap \n_{\varphi}.
	\end{equation}
	From \eqref{eq:Parseval's relation} it follows that
	\begin{equation*}
	\widehat{\varphi}(\lambda(\varphi_x)^*\lambda(\varphi_x))
	=\langle\widehat{\Lambda}(\lambda(\varphi_x)),\widehat{\Lambda}(\lambda(\varphi_x))\rangle
	=\varphi(x^*x),~~x\in L_1(\G,\varphi)\cap \n_{\varphi}.
	\end{equation*}
	So we have 
	\begin{equation*}
	\|\F_2(x)\|_{L_2(\widehat{\G},\widehat{\varphi})}=\|x\|_{L_2(\G,\varphi)},~~x\in L_1(\G,\varphi)\cap \n_{\varphi}.
	\end{equation*}
	Since $L_1(\G,\varphi)\cap \n_{\varphi}$ is dense in $L_2(\G,\varphi)$, $\F_2$ can be extended to an isometry from $L_2(\G,\varphi)$ to $L_2(\widehat{\G},\widehat{\varphi})$, which we still denote by $\F_2$. 
	
	Now we may define an operator $\F$ on $L_1(\G,\varphi)+ L_2(\G,\varphi)$ as $\F(x)=\F_1(x_1)+\F_2(x_2)$, where $x=x_1+x_2$ with $x_i\in L_i(\G,\varphi),i=1,2$. One can check that it is well-defined and $\F|_{L_i(\G,\varphi)}=\F_i,i=1,2$. Thus the general $L_p$-Fourier transform $\F_p,1<p<2$, is defined to be the restriction of $\F$ to $L_p(\G,\varphi)\subset L_1(\G,\varphi)+ L_2(\G,\varphi)$. By complex interpolation, we have the \emph{Hausdorff--Young inequality}:
	\begin{equation}\label{ineq:HY}
	\|\F_p(x)\|_{L_{p'}(\widehat{\G},\widehat{\varphi})}\le \|x\|_{L_p(\G,\varphi)},~~x\in L_p(\G,\varphi),
	\end{equation}
	where $1\le p\le 2$ and $1/p+1/p'=1$. If we use real interpolation instead of complex interpolation, we get 
	\begin{equation}\label{ineq:stronger HY}
	\|\F(x)\|_{L_{p'}(\widehat{\G},\widehat{\varphi})}
	\precsim_{p} \|x\|_{L_{p,p'}(\G,\varphi)},~~x\in L_{p,p'}(\G,\varphi).
	\end{equation}
	Compared with \eqref{ineq:HY}, the constant $c_p$ in \eqref{ineq:stronger HY} is worse, but the space $L_{p,p'}(\G,\varphi)$ is larger than $L_{p}(\G,\varphi)$ when $1\le p<2$.
	
	\begin{defn}
		For any $x\in L_0(\G,\varphi)$, we call $m_x$ an \emph{$L_p$-$L_q$ Fourier multiplier} if the map $\F(y)\mapsto\F(xy)$ is well-defined and extends to bounded map from $L_p(\widehat{\G},\widehat{\varphi})$ to $L_q(\widehat{\G},\widehat{\varphi})$. One may also consider the map $\F(y)\mapsto\F(yx)$, which is similar.  
	\end{defn}

	\subsection{The dual/inverse Fourier transform}
	On the dual quantum group one can also define the Fourier transform $\widehat{\F}:L_1(\widehat{\G},\widehat{\varphi})+L_2(\widehat{\G},\widehat{\varphi})\to L_\infty(\G,\varphi)+L_2(\G,\varphi)$, whose restriction to $L_1(\widehat{\G},\widehat{\varphi})$ is $\widehat{\lambda}\circ\widehat{j}^{-1}$, where $\widehat{j}:L_1(\widehat{\G},\widehat{\varphi})\to \widehat{\M}_*,x\mapsto \widehat{\varphi}(\cdot x)$. Then $\widehat{\F}_2$ is the inverse of $\F_2$.

	\begin{prop}\label{prop:inverse fourier transform}
	Let $\G=(\M,\Delta,\varphi,\psi)$ be a locally compact quantum group with dual $\widehat{\G}=(\widehat{\M},\widehat{\Delta},\widehat{\varphi},\widehat{\psi})$. Then we have
	\begin{enumerate}
		\item [(1)] $\widehat{\F}(\F(x))=x,~~x\in L_2(\G,\varphi)$;
		\item [(2)] $\F(\widehat{\F}(a))=a,~~a\in L_2(\widehat{\G},\widehat{\varphi})$.
	\end{enumerate}
\end{prop} 

\begin{proof}
	Note that the inclusion map $\Lambda:\n_{\varphi}\to H_{\varphi}=L_2(\G,\varphi)$ can be extended to the whole Hilbert space $L_2(\G,\varphi)$. We shall still use $\Lambda$ to denote its extension. The same goes to $\widehat{\Lambda}$. Recall that
	$$\langle \Lambda(x),\Lambda(y)\rangle
	=\langle \widehat{\Lambda}(\F(x)),\Lambda(y) \rangle,~~x\in L_1(\G,\varphi)\cap \n_{\varphi},y\in \n_{\varphi}.$$
	Since $L_1(\G,\varphi)\cap \n_{\varphi}$ is dense in $L_2(\G,\varphi)$, for any $x\in L_2(\G,\varphi)$ we may choose a net $\{x_{\alpha}\}_{\alpha}\subset L_1(\G,\varphi)\cap \n_{\varphi}$ such that $\lim\limits_{\alpha}x_{\alpha}=x$ in $\n Zn_{\varphi}$. $\F|_{\n_{\varphi}}$ is an isometry, so we obtain
	$$\langle \Lambda(x),\Lambda(y)\rangle
	=\lim\limits_{\alpha}\langle \Lambda(x_{\alpha}),\Lambda(y)\rangle
	=\lim\limits_{\alpha}\langle \widehat{\Lambda}(\F(x_{\alpha})),\Lambda(y) \rangle
	=\langle \widehat{\Lambda}(\F(x)),\Lambda(y) \rangle,~~y\in \n_{\varphi}.$$
    Hence $\widehat{\Lambda}(\F(x))=\Lambda(x), x\in L_2(\G,\varphi).$
    Since $\F(x)\in L_2(\widehat{\G},\widehat{\varphi})$ for all $x\in L_2(\G,\varphi)$, we have
	$$\Lambda(\widehat{\F}(\F(x)))=\widehat{\widehat{\Lambda}}(\widehat{\F}(\F(x)))=\widehat{\Lambda}(\F(x))=\Lambda(x),~~x\in L_2(\G,\varphi).$$
	Hence $\widehat{\F}(\F(x))=x$. This proves (1). The proof of (2) is similar.
\end{proof}

	Since $\widehat{\F}$ is the Fourier transform on $\widehat{\G}$, we have 
	\begin{equation*}
	\|\widehat{\F}(a)\|_{L_{\infty}(\G,\varphi)}
    \le \|a\|_{L_1(\widehat{\G},\widehat{\varphi})},~~a\in L_1(\widehat{\G},\widehat{\varphi}),
	\end{equation*}
	This, together with Proposition \ref{prop:inverse fourier transform}, yields
	\begin{equation}\label{ineq:contraction of inverse of F1}
	\|x\|_{L_{\infty}(\G,\varphi)}
	\le \|\F(x)\|_{L_1(\widehat{\G},\widehat{\varphi})},
	\end{equation}
	for all $x\in L_2(\G,\varphi)$ such that $\F(x)\in L_1(\widehat{\G},\widehat{\varphi})$, or equivalently, for all $x$ such that $\F(x)\in L_1(\widehat{\G},\widehat{\varphi})\cap L_2(\widehat{\G},\widehat{\varphi})$. Since $L_1(\widehat{\G},\widehat{\varphi})\cap L_2(\widehat{\G},\widehat{\varphi})$ is dense in $L_1(\widehat{\G},\widehat{\varphi})$, the map $\F(x)\mapsto x$ can be extended to a contraction from $L_1(\widehat{\G},\widehat{\varphi})$ to $L_{\infty}(\G,\varphi)$. Recall that 
	\begin{equation}\label{eq:isometry of F2}
	\|x\|_{L_2(\G,\varphi)}=\|\F(x)\|_{L_2(\widehat{\G},\widehat{\varphi})},~~x\in L_2(\G,\varphi).
	\end{equation}
	Combing \eqref{ineq:contraction of inverse of F1}, \eqref{eq:isometry of F2}, and applying real interpolation, we get
	\begin{equation}\label{ineq:stronger HY for inverse}
	\|x\|_{L_{p',p}(\G,\varphi)}
	\precsim_{p} \|\F(x)\|_{L_p(\widehat{\G},\widehat{\varphi})},
	\end{equation}
	for all $x$ such that $\F(x)\in L_p(\widehat{\G},\widehat{\varphi})$.
	
	\section{The proofs and examples}\label{sec:proof}
	
	\subsection{Fourier multipliers}
	This subsection is devoted to the proofs of our results for Fourier multipliers. Some examples will also be presented. In the following we shall simply use $\|\cdot\|_{p,q}$ to denote $\|\cdot\|_{L_{p,q}(\G,\varphi)}$ or $\|\cdot\|_{L_{p,q}(\widehat{\G},\widehat{\varphi})}$ whenever no ambiguity can occur. 
	
	\begin{proof}[Proof of Theorem \ref{thm:main reulst Lp-Lq multiplier}]
		Note that $1/q'=1/r+1/p'$. Then for any $x\in L_{r,\infty}(\G,\varphi)$ and $y\in L_1(\G,\varphi)+L_2(\G,\varphi)$ such that $\F(y)\in L_p(\widehat{\G},\widehat{\varphi})$, we have
		\begin{equation*}
		\|\F(xy)\|_q
		\stackrel{\eqref{ineq:stronger HY}}{\precsim_{q}} \|xy\|_{q',q}
		\stackrel{\eqref{ineq:Holder}}{\precsim_{q}} \|x\|_{r,\infty}\|y\|_{p',q}
		\stackrel{\eqref{ineq:lorentz space}}{\precsim_{p,q}} \|x\|_{r,\infty}\|y\|_{p',p}
		\stackrel{\eqref{ineq:stronger HY for inverse}}{\precsim_{p,q}} \|x\|_{r,\infty}\|\F(y)\|_{p}.\qedhere
		\end{equation*}
	\end{proof}

\begin{rem}\label{rem:Fourier-complex}
	From the proof, one can see that the result can be extended to the boundedness of Fourier multipliers between more general Lorentz spaces, which is beyond the aim of this paper. Also, if we use complex interpolation instead of real interpolation, i.e. the usual Hausdorff--Young inequalities, then one may get an upper bound of $\|x\|_r$ instead of $c_{p,q}\|x\|_{r,\infty}$. Details are provided for the Schur multipliers. See Remark \ref{rem:Schur-complex}.
\end{rem}

	\begin{proof}[Proof of Theorem \ref{thm:main result-Paley type}]
		Note that $1/p=1/p'+1/s$. For any $a\in L_{s,\infty}(\widehat{\G},\widehat{\varphi})$ and any $x\in L_p(\G,\varphi)$, we have 
		\begin{equation*}
		\|a\F(x)\|_{p}
		=\|a\F(x)\|_{p,p}
		\stackrel{\eqref{ineq:Holder}}{\precsim_{p}} \|a\|_{s,\infty}\|\F(x)\|_{p',p}
		\stackrel{\eqref{ineq:stronger HY}}{\precsim_{p}} \|a\|_{s,\infty}\|x\|_{p}.\qedhere
		\end{equation*}
	\end{proof}
	
	\begin{proof}[Proof of Theorem \ref{thm:main result-Lp multiplier}]
		This is a direct consequence of Theorem \ref{thm:main reulst Lp-Lq multiplier}. Indeed, since $h$ is a state, we have by H\"older's inequality that $\|x\|_{p}\le \|x\|_{q}$ whenever $x\in L_q(\G,h)$ and $p\le q$. Thus for any $1<p\le 2\le q<\infty$ we have 
		\begin{equation}\label{ineq:lp-lq discrete}
		\|\widehat{\F}(ab)\|_{p}
		\le \|\widehat{\F}(ab)\|_{q}
		\precsim_{p,q} \|a\|_{r,\infty}\|\widehat{\F}(b)\|_{p}
		\precsim_{p,q} \|a\|_{r,\infty}\|\widehat{\F}(b)\|_{q},
		\end{equation}
		for all $a\in \ell_{r,\infty}(\widehat{\G},\widehat{h})$ and $\widehat{\F}(b)\in L_q(\G,h)$. The first two inequalities of \eqref{ineq:lp-lq discrete} imply that $m_a$ is an $L_p$-Fourier multiplier:
		\begin{equation}\label{ineq:Lp-multiplier}
		\|m_a:L_p(\G,h)\to L_p(\G,h)\|\precsim_{p,q} \|a\|_{r,\infty},
		\end{equation}
		while the last two inequalities of \eqref{ineq:lp-lq discrete} imply that $m_a$ is an $L_q$-Fourier multiplier:
		\begin{equation}\label{ineq:Lq-multiplier}
		\|m_a:L_q(\G,h)\to L_q(\G,h)\|\precsim_{p,q} \|a\|_{r,\infty},
		\end{equation} 
		We may choose $q=2$ in \eqref{ineq:Lp-multiplier} and $p=2$ in \eqref{ineq:Lq-multiplier}. Hence for any $1<p<\infty$ we have
		\begin{equation*}
		\|m_a:L_p(\G,h)\to L_p(\G,h)\|\precsim_{p} \|a\|_{p^*,\infty},
		\end{equation*}
		with $1/p^*=|1/2-1/p|$.
	\end{proof}

\begin{rem}
	The index $r$ in Theorem \ref{thm:main reulst Lp-Lq multiplier} is sharp in general. To see this, take $\G=\Z$ with $\widehat{\G}=\T$. By Theorem \ref{thm:main reulst Lp-Lq multiplier} we have for $1<p\le 2\le q<\infty$ that
	\begin{equation}\label{ineq:Lp-Lq multiplier on T}
	\|m_{\phi}:L_p(\T)\to L_q(\T)\|\precsim_{p,q} \|\phi\|_{\ell_{r,\infty}(\Z)},
	\end{equation}
	where $1/r=1/p-1/q$. Indeed, by \cite[Lemma 6.6, page 129, Vol. II]{Zygmund02}, for any $1<p<\infty$ and Fourier series 
	$$f(x):=\sum_{n=1}^{\infty}a_n \cos (nx)=\frac{1}{2}\sum_{n\in \Z}a_{|n|}e^{inx},$$
	such that $a_n\downarrow 0$ as $n\to\infty$, we have 
	\begin{equation}\label{characterization of function in Lp}
	f\in L_p(\T)\textnormal{ if and only if } \sum_{n\ge 1}n^{p-2}a^p_n<\infty.
	\end{equation}
	Now suppose that $r$ in \eqref{ineq:Lp-Lq multiplier on T} can  be replaced by some $s>r$. Consider $\phi(n):=|n|^{-\frac{1}{s}},n\neq 0$ and $\phi(0):=0$. It is easy to see that $\phi\in \ell_{s,\infty}(\Z)\setminus \ell_{r,\infty}(\Z)$. Set $\alpha:=1/r-1/s>0$ and $a_n:=n^{\frac{1}{p}-1-\alpha}$. Since 
	$$p-2+p\left(\frac{1}{p}-1-\alpha\right)=-1-p\alpha<-1,$$
	$$q-2+q\left(\frac{1}{p}-1-\alpha-\frac{1}{s}\right)=q-2+q\left(\frac{1}{q}-1\right)=-1,$$
	we have  
	$$\sum_{n\ge 1}n^{p-2}a^p_n=\sum_{n\ge 1}n^{-1-p\alpha}<\infty,~~
	\sum_{n\ge 1}n^{q-2}(a_n \phi(n))^q=\sum_{n\ge 1}n^{-1}=\infty.$$
	By \eqref{characterization of function in Lp}, $f\in L_p(\T)$ while $m_{\phi}(f)\notin L_q(\T)$, which leads to a contradiction. So $r$ is sharp.
\end{rem}

\begin{rem}\label{rem:endpoint}
	The result of Corollary \ref{cor:group vna}  may fail in the endpoint case $p=1$. I am very grateful to \'Eric Ricard for pointing this out to me, and allowing me to include his proof here.  Take $G=\mathbb{Z}$ and $\widehat{G}=\mathbb{T}$. Then there exists $\phi:\mathbb{Z}\to \mathbb{R}$ such that $\phi\in\ell_{2,\infty}(\mathbb{Z})$ while the Fourier multiplier $m_\phi$ is unbounded on $L_1(\mathbb{T})$. To see this, take 
	\begin{equation*}
	\phi(n)=
	\begin{cases}
	\frac{1}{\sqrt{k}}&n=2^k, k\ge 1\\
	0&\text{otherwise}
	\end{cases}.
	\end{equation*}
	Clearly $\phi\in\ell_{2,\infty}(\mathbb{Z})\setminus \ell_2(\mathbb{Z})$. Suppose that the Fourier multiplier $m_\phi$  is bounded over $L_1(\mathbb{T})$. Then there exist a measure $\mu$ on $\mathbb{T}$ such that $m_\phi(f)=\mu\ast f$, with $\ast$ being the convolution. Since 
	$$\widehat{\mu}(n)=\phi(n)=0,~~n<0,$$
	by F. and M. Riesz theorem \cite[Theorem 17.13, page 341]{Rudin87real}, $\mu$ is absolutely continuous with respect to the Lebesgue measure $d\theta$. So $m_\phi$ is a convolution operator, i.e. $m_\phi(f)=h\ast f$ for some $h\in L_1(\mathbb{T})$ such that $\widehat{h}=\phi$. By construction, $\phi$ is supported on a Lacunary set \{$2^k,k\ge 1$\}.  Hence we have \cite[Theorem 3.6.4]{Grafakos2014classical}
	\begin{equation}\label{ineq:contradition}
	\|\phi\|_{\ell_2(\mathbb{Z})}=\|h\|_{L_2(\mathbb{T})}\le K\|h\|_{L_1(\mathbb{T})},
	\end{equation}
	for some constant $K>0$. However, the left hand side is unbounded as $\phi\notin \ell_2(\mathbb{Z})$. This leads to a contradiction. Therefore Corollary \ref{cor:group vna}  fails when $p=1$.
\end{rem}

	\begin{example}
		Let $G$ be a finitely generated group with the unit $e$ and a symmetric set $S$ of generators. By saying symmetric we mean $x^{-1}\in S$ whenever $x\in S$. Then it has an \emph{exponential growth}, i.e.,
		\begin{equation}\label{ineq:exponential growth}
		|\{x\in G:d(x,e)\le n\}|\le M^n,~~n\ge 1,
		\end{equation}
		for some $M>1$, where $d$ is the word metric on $G$ with respect to $S$ and $|\cdot|$ denotes the counting measure on $G$. Indeed, one can always choose $M$ to be $|S|$. Then for any $\phi:G\to\mathbb{C}$ such that $|\phi(g)|\le C M^{-\frac{|g|}{p^*}}$, where $|g|:=d(g,e)$ and $C>0$ is a constant. Then
		\begin{equation*}
		|\phi(g)|\ge \alpha \textnormal{ implies }|g|\le -p^*\log_{M} \frac{\alpha}{C},~~\alpha>0.
		\end{equation*}
		Therefore
		\begin{equation*}
		\alpha^{p^*}|\{g\in G:|\phi(g)|\ge \alpha\}|
		\le \alpha^{p^*} M^{-p^*\log_{M} \frac{\alpha}{C}} 
		\le C^{p^*}<\infty,~~\alpha>0,
		\end{equation*}
		and we have $\phi\in \ell_{p^*,\infty}(G)$, whence $m_{\phi}$ is an $L_p$-multiplier on $L_p(\widehat{G},\tau)$. For free group on $N$ generators $\mathbb{F}_N$, we may choose $S$ as the set consisting of $N$ generators with their inverses and let $M=2N$.
		
		If moreover, $G$ is of \emph{polynomial growth}, i.e. the right hand side of \eqref{ineq:exponential growth} can be replaced by some polynomial $p(n)$, or equivalently, $n^k$ for some $k>0$, then a similar argument yields that for any $\phi:G\to\mathbb{C}$ such that $|\phi(g)|\le C|g|^{-\frac{k}{p^*}}$, we have $\phi\in \ell_{p^*,\infty}(G)$, and then $m_{\phi}$ is an $L_p$-Fourier multiplier on $L_p(\widehat{G},\tau)$.
	\end{example}
	
	\subsection{Schur multipliers}
	In this subsection we prove Theorem \ref{thm:schur} for $\S_p$-$\S_q$ Schur multipliers. Recall that the Schatten $p$-class $\S_p(H)$ is the noncommutative $L_p$-space $L_p(B(H),\text{Tr})$ with Tr being the usual trace. For any set $X$, any $a=(a_{ij})_{i,j\in X}$ induces a \emph{Schur multiplier} $A$ given by $A(x_{ij})=(a_{ij}x_{ij})$. Here we are interested in $\S_p$-$\S_q$ boundedness of $A$. In the following we use $\|\cdot\|_p$ to denote the Schatten $p$-norms. Note first that we have
	\begin{equation}\label{ineq:schur1-infty}
	\|x\|_{\infty}\le \|x\|_{\ell_1(X\times X)},
	\end{equation}
	and 
	\begin{equation}\label{eq:paserval-schur}
	\|x\|_{2}=\|x\|_{\ell_2(X\times X)}.
	\end{equation}
	With \eqref{ineq:schur1-infty} and \eqref{eq:paserval-schur}, the complex interpolation gives 
	\begin{equation}\label{ineq:HY_Schur}
	\|x\|_{p'}\le \|x\|_{\ell_{p}(X\times X)},~~1< p< 2,
	\end{equation}
	while the real interpolation implies 
	\begin{equation}\label{ineq:stronger HY_Schur}
	\|x\|_{p'}\precsim_{p} \|x\|_{\ell_{p,p'}(X\times X)},~~1< p< 2.
	\end{equation}
	Similarly, from 
	\begin{equation}
	\|x\|_{\ell_\infty(X\times X)}\le \|x\|_{1},~~\|x\|_{\ell_2(X\times X)}=\|x\|_{2},
	\end{equation}
	we have by complex interpolation that 
	\begin{equation}\label{ineq:HY_Schur_dual}
	 \|x\|_{\ell_{p'}(X\times X)}\le \|x\|_{p},~~1< p< 2,
	\end{equation}
	and real interpolation that 
	\begin{equation}\label{ineq:stronger HY_Schur_dual}
	\|x\|_{\ell_{p',p}(X\times X)}\precsim_{p}\|x\|_{p},~~1< p< 2.
	\end{equation}
	
	\begin{proof}[Proof of Theorem \ref{thm:schur}]
		For Schur multipliers $A$ induced by $a=(a_{ij})_{i,j\in X}$, we have for any $x=(x_{ij})_{i,j\in X}\in \S_p(\ell_{2}(X))$ that 
		\begin{align*}
		\|Ax\|_q
		&\stackrel{\eqref{ineq:stronger HY_Schur}}{\precsim_{q}}\|(a_{ij}x_{ij})\|_{\ell_{q',q}(X\times X)}\\
		&\stackrel{\eqref{ineq:Holder}}{\precsim_{q}}\|a\|_{\ell_{r,\infty}(X\times X)}\|x\|_{\ell_{p',q}(X\times X)}\\
		&\stackrel{\eqref{ineq:lorentz space}}{\precsim_{p,q}}\|a\|_{\ell_{r,\infty}(X\times X)}\|x\|_{\ell_{p',p}(X\times X)}\\
		&\stackrel{\eqref{ineq:stronger HY_Schur_dual}}{\precsim_{p,q}}\|a\|_{\ell_{r,\infty}(X\times X)}\|x\|_{p}.\qedhere
		\end{align*}
	\end{proof}

	\begin{rem}\label{rem:Schur-complex}
		If we use complex interpolation instead of real interpolation, we get
		\begin{align*}
		\|Ax\|_q
		\stackrel{\eqref{ineq:HY_Schur}}{\le}\|(a_{ij}x_{ij})\|_{\ell_{q'}(X\times X)}
		\stackrel{\text{H\"older}}{\le}\|a\|_{\ell_{r}(X\times X)}\|x\|_{\ell_{p'}(X\times X)}
		\stackrel{\eqref{ineq:HY_Schur_dual}}{\le}\|a\|_{\ell_{r}(X\times X)}\|x\|_{p}.
		\end{align*}
	\end{rem}

	\subsection{Remarks}
	Our proof uses the following interpolation result: for $1\le p_0<p_1\le \infty,0<\theta <1$, and $1/p=(1-\theta)/p_0+\theta/p_1$, we have 
	\begin{equation*}
	(L_{p_0}(\M,\varphi),L_{p_1}(\M,\varphi))_{\theta,p}=L_p(\M,\varphi) ~~(\text{with equivalent norms}),
	\end{equation*}
	when $\varphi$ is a trace. However, when $\varphi$ is a weight, this fails in general (\cite[Section 3]{PisierXu2003LP}). That is why we assume the left Haar weight $\varphi$ and its dual $\widehat{\varphi}$ to be tracial. If we use complex interpolation instead of real interpolation, then one can still get an upper bound of $\|x\|_{L_{r}(\G,\varphi)}$ in Theorem \ref{thm:main reulst Lp-Lq multiplier} for general locally compact quantum groups. See Remarks \ref{rem:Fourier-complex} and \ref{rem:Schur-complex}. Certainly in this case the definition of Fourier multipliers is more involved. 
	
	 We end with the following interesting question. Let $1<p\le 2\le q<\infty$ and $1/r=1/p-1/q$. Suppose that $G$ is a locally compact non-unimodular group with $\mu$ being the left Haar measure. Let $\widehat{G}$ be the dual of $G$ with $\widehat{\varphi}$ being the dual left Haar weight. Then for the Fourier multiplier $m_\phi$ with the symbol $\phi\in L_{r,\infty}(G,\mu)$, do we have
	\begin{equation*}
	\|m_\phi: L_p(\widehat{G},\widehat{\varphi})\to L_q(\widehat{G},\widehat{\varphi})\|\precsim_{p,q}\|\phi\|_{L_{r,\infty}(G,\mu)} ? 
	\end{equation*}
	Here $\varphi=\mu$ is tracial, while $\widehat{\varphi}$ is not. One may choose various equivalent ways to define $L_p(\widehat{G},\widehat{\varphi})$, and the definition of Fourier multiplier $m_\phi$ needs to be suitably adapted accordingly. 
	
	\subsection*{Acknowledgement}
	Part of this project was finished during my PhD.  I am very grateful to my supervisor Professor Quanhua Xu for bringing me the topic and fruitful discussions. Part of this work was carried out during a visit to Caen. I would like to thank \'Eric Ricard for the invitation, warm hospitality, and helpful discussions. In particular, Remark \ref{rem:endpoint} is from him and I am very grateful for allowing me to include it here. Many thanks to Adam Skalski and Simeng Wang for helpful discussions and valuable comments. The research was partially supported by the NCN (National Centre of Science) grant 2014/14/E/ST1/00525, the French project ISITE-BFC (contract ANR-15-IDEX-03) and Lise Meitner fellowship, Austrian Science Fund (FWF) M3337.


\begin{thebibliography}{AMR18}
	
	\bibitem[AMR18]{Ruzhansky2018compactquantumgroup}
	R.~Akylzhanov, S.~Majid, and M.~Ruzhansky.
	\newblock Smooth dense subalgebras and {F}ourier multipliers on compact quantum
	groups.
	\newblock {\em Comm. Math. Phys.}, 362(3):761--799, 2018.
	
	\bibitem[ANR16]{Ruzhansky2016liegroup}
	R.~Akylzhanov, E.~D. Nursultanov, and M.~V. Ruzhanski\u\i.
	\newblock Hardy-{L}ittlewood-{P}aley-type inequalities on compact {L}ie groups.
	\newblock {\em Mat. Zametki}, 100(2):287--290, 2016.
	
	\bibitem[AR16]{Ruzhansky2016locallycompact}
	R.~Akylzhanov and M.~Ruzhansky.
	\newblock Fourier multipliers and group von {N}eumann algebras.
	\newblock {\em C. R. Math. Acad. Sci. Paris}, 354(8):766--770, 2016.
	
	\bibitem[BL76]{BerghLofstrom1976interpolationspaces}
	J.~Bergh and J.~L\"ofstr\"om.
	\newblock {\em Interpolation spaces. {A}n introduction}.
	\newblock Springer-Verlag, Berlin-New York, 1976.
	\newblock Grundlehren der Mathematischen Wissenschaften, No. 223.
	
	\bibitem[Cas13]{Caspers2013fouriertransform}
	M.~Caspers.
	\newblock The {$L^p$}-{F}ourier transform on locally compact quantum groups.
	\newblock {\em J. Operator Theory}, 69(1):161--193, 2013.
	
	\bibitem[Cas17]{Caspers2017notes}
	M.~Caspers.
	\newblock Locally compact quantum groups.
	\newblock In {\em Topological quantum groups}, volume 111 of {\em Banach Center
		Publ.}, pages 153--184. Polish Acad. Sci. Inst. Math., Warsaw, 2017.
	
	\bibitem[Con80]{Connes1980spatial}
	A.~Connes.
	\newblock On the spatial theory of von {N}eumann algebras.
	\newblock {\em J. Funct. Anal.}, 35(2):153--164, 1980.
	
	\bibitem[Coo10]{Cooney2010hausdorffyoung}
	T.~Cooney.
	\newblock A {H}ausdorff-{Y}oung inequality for locally compact quantum groups.
	\newblock {\em Internat. J. Math.}, 21(12):1619--1632, 2010.
	
	\bibitem[Dix53]{Dixmier1953}
	J.~Dixmier.
	\newblock Formes lin\'eaires sur un anneau d'op\'erateurs.
	\newblock {\em Bull. Soc. Math. France}, 81:9--39, 1953.
	
	\bibitem[ES92]{ES92kac}
	M.~Enock and J.-M. Schwartz.
	\newblock {\em Kac algebras and duality of locally compact groups}.
	\newblock Springer-Verlag, Berlin, 1992.
	\newblock With a preface by Alain Connes, With a postface by Adrian Ocneanu.
	
	\bibitem[FK86]{FackKosaki1986sigularnumbers}
	T.~Fack and H.~Kosaki.
	\newblock Generalized {$s$}-numbers of {$\tau$}-measurable operators.
	\newblock {\em Pacific J. Math.}, 123(2):269--300, 1986.
	
	\bibitem[Gra14]{Grafakos2014classical}
	L.~Grafakos.
	\newblock {\em Classical {F}ourier analysis}, volume 249 of {\em Graduate Texts
		in Mathematics}.
	\newblock Springer, New York, third edition, 2014.
	
	\bibitem[Hil81]{Hilsum1981}
	M.~Hilsum.
	\newblock Les espaces {$L^{p}$} d'une alg\`ebre de von {N}eumann d\'efinies par
	la deriv\'ee spatiale.
	\newblock {\em J. Funct. Anal.}, 40(2):151--169, 1981.
	
	\bibitem[H{\"o}r60]{Hormander1960}
	L.~H{\"o}rmander.
	\newblock Estimates for translation invariant operators in {$L^{p}$}\ spaces.
	\newblock {\em Acta Math.}, 104:93--140, 1960.
	
	\bibitem[Kun58]{Kunze1958fouriertransform}
	R.~A. Kunze.
	\newblock {$L_{p}$} {F}ourier transforms on locally compact unimodular groups.
	\newblock {\em Trans. Amer. Math. Soc.}, 89:519--540, 1958.
	
	\bibitem[KV00]{KustermansVaes2000lcqg}
	J.~Kustermans and S.~Vaes.
	\newblock Locally compact quantum groups.
	\newblock {\em Ann. Sci. \'Ecole Norm. Sup. (4)}, 33(6):837--934, 2000.
	
	\bibitem[KV03]{KV03}
	J.~Kustermans and S.~Vaes.
	\newblock Locally compact quantum groups in the von {N}eumann algebraic
	setting.
	\newblock {\em Math. Scand.}, 92(1):68--92, 2003.
	
	\bibitem[PX03]{PisierXu2003LP}
	G.~Pisier and Q.~Xu.
	\newblock Non-commutative {$L^p$}-spaces.
	\newblock In {\em Handbook of the geometry of {B}anach spaces, {V}ol.\ 2},
	pages 1459--1517. North-Holland, Amsterdam, 2003.
	
	\bibitem[Rud87]{Rudin87real}
	W.~Rudin.
	\newblock {\em Real and complex analysis}.
	\newblock McGraw-Hill Book Co., New York, third edition, 1987.
	
	\bibitem[Seg53]{Segal1953abstractintergration}
	I.~E. Segal.
	\newblock A non-commutative extension of abstract integration.
	\newblock {\em Ann. of Math. (2)}, 57:401--457, 1953.
	
	\bibitem[Ter17]{Terp2017fouriertransform}
	M.~Terp.
	\newblock {$L^p$} {F}ourier transformation on non-unimodular locally compact
	groups.
	\newblock {\em Adv. Oper. Theory}, 2(4):547--583, 2017.
	
	\bibitem[Wan95]{Wan95}
	S.~Wang.
	\newblock Free products of compact quantum groups.
	\newblock {\em Comm. Math. Phys.}, 167(3):671--692, 1995.
	
	\bibitem[Wan98]{Wan98}
	S.~Wang.
	\newblock Quantum symmetry groups of finite spaces.
	\newblock {\em Comm. Math. Phys.}, 195(1):195--211, 1998.
	
	\bibitem[Wor98]{Woronowicz1998notes}
	S.~L. Woronowicz.
	\newblock Compact quantum groups.
	\newblock In {\em Sym\'etries quantiques ({L}es {H}ouches, 1995)}, pages
	845--884. North-Holland, Amsterdam, 1998.
	
	\bibitem[You18]{Sang-Gyun2018hardylittlewood}
	S.-G. Youn.
	\newblock Hardy-{L}ittlewood inequalities on compact quantum groups of {K}ac
	type.
	\newblock {\em Anal. PDE}, 11(1):237--261, 2018.
	
	\bibitem[Zyg02]{Zygmund02}
	A.~Zygmund.
	\newblock {\em Trigonometric series. {V}ol. {I}, {II}}.
	\newblock Cambridge Mathematical Library. Cambridge University Press,
	Cambridge, third edition, 2002.
	\newblock With a foreword by Robert A. Fefferman.
	
\end{thebibliography}
\end{document}